\documentclass[reqno,11pt]{amsart}
\usepackage{a4wide,color,eucal,enumerate,mathrsfs}
\usepackage[normalem]{ulem}
\usepackage{amsmath,amssymb,epsfig,amsthm} 
\usepackage[latin1]{inputenc}
\usepackage{psfrag}

\def\nn{{\mathbb N}}

\def\dist{{\mathop\mathrm{\,dist\,}}}

\def\ez{\epsilon}

\def\boz{{\Omega}}

\def\wz{\widetilde}

\def\ls{\lesssim}

\def\bint{{\ifinner\rlap{\bf\kern.35em--}
\int\else\rlap{\bf\kern.45em--}\int\fi}\ignorespaces}

\def\bbint{{\ifinner\rlap{\bf\kern.35em--}
\hspace{0.078cm}\int\else\rlap{\bf\kern.45em--}\int\fi}\ignorespaces}

\def\esup{\mathop\mathrm{\,esssup\,}}

\def\diam{{\mathop\mathrm{\,diam\,}}}

\newtheorem{thm}{Theorem}[section]
\newtheorem{lem}[thm]{Lemma}
\newtheorem{cor}[thm]{Corollary}
\newtheorem{defn}[thm]{Definition}
\numberwithin{equation}{section}

\theoremstyle{remark}

\def\bint{{\ifinner\rlap{\bf\kern.35em--}
\int\else\rlap{\bf\kern.45em--}\int\fi}\ignorespaces}

\begin{document}

\title[A density problem for Sobolev spaces on Gromov hyperbolic domains]
{A density problem for Sobolev spaces on Gromov hyperbolic domains}

\author{Pekka Koskela}
\author{Tapio Rajala}
\author{Yi Ru-Ya Zhang}

\address{Department of Mathematics and Statistics \\
         P.O. Box 35 (MaD) \\
         FI-40014 University of Jyv\"as\-kyl\"a \\
         Finland}
\email{pekka.j.koskela@jyu.fi} 
\email{tapio.m.rajala@jyu.fi}
\email{yi.y.zhang@jyu.fi}

\thanks{All authors partially supported by the Academy of Finland.}
\date{\today}



\begin{abstract}
We prove that for a bounded domain $\Omega\subset \mathbb R^n$ which is 
Gromov hyperbolic with respect to the quasihyperbolic metric, especially 
when $\Omega$ is a finitely connected planar domain, the Sobolev space 
$W^{1,\,\infty}(\Omega)$ is dense in $W^{1,\,p}(\Omega)$ for any $1\le p<\infty$. 
Moreover if $\Omega$ is also Jordan or quasiconvex, 
then $C^{\infty}(\mathbb R^n)$ is dense in $W^{1,\,p}(\Omega)$  for 
$1\le p<\infty$. 
\end{abstract}


\maketitle

\tableofcontents

\section{Introduction}

Let $\boz \subset\mathbb R^n$ be a domain with $n\ge 2$. We denote by 
$D_i u=\frac{\partial u}{\partial x_i}$ the (weak) $i^{th}$ partial 
derivative of a 
locally integrable function $u$, and by 
$\nabla u=(D_1 u,\,\dots,\,D_n u)$ the (weak) gradient. 
Then for $1 \le p \le \infty$ we define  the Sobolev space as 
$$W^{1,\,p}(\Omega)=\left\{u \in L^p(\boz)\mid D_i u \in L^p(\boz),\, 1\le i\le n \right\}, $$
with the norm
$$\|u\|^p_{W^{1,p}(\boz)} = \int_{\boz} |u(x)|^p+|\nabla u(x)|^p \, dx$$
for $1 \le p < \infty$, and
$$\|u\|_{W^{1,\infty}(\boz)} =\esup_{x\in\Omega} |u(x)|+ \sum_{1\le i\le n} \esup_{x\in\Omega}  |D_i u(x)|. $$

It is a fundamental property of Sobolev spaces that smooth 
functions defined in $\Omega$ are dense in 
$W^{1,p}(\Omega)$ for any domain $\Omega\subset \mathbb R^n$ when 
$1 \le p < \infty$. 
If each function in $W^{1,p}(\Omega)$ is the restriction of a function
in  $W^{1,\,p}(\mathbb R^n),$ one can then obviously use global smooth 
functions to approximate functions in $ W^{1,p}(\Omega)$. 
This is in particular the case for Lipschitz domains. 
Moreover,  if $\Omega$ satisfies the so-called ``segment condition", then one 
has that $C^{\infty}(\mathbb R^n)$ is dense in $W^{1,\,p}(\Omega)$; see e.g. 
\cite{AF2003} for  references. 

In the planar setting,  Lewis proved in \cite{L1985} that 
$C^{\infty}(\mathbb R^2)$ is dense in $W^{1,p}(\Omega)$ for $1<p<\infty$ 
provided that $\Omega$ is a Jordan domain.
More recently, in \cite{GT2007} it was shown by Giacomini and Trebeschi that, 
for bounded simply connected planar domains, $W^{1,\,2}(\Omega)$ is dense in 
$W^{1,\,p}(\Omega)$ for all $1\le p <2$. 
Motivated by the results above, Koskela and  Zhang proved in \cite{KZ2016}
that for any bounded simply connected domain and any $1\le p<\infty$, $W^{1,\,\infty}(\Omega)$ is dense in $W^{1,\,p}(\Omega)$, and $C^{\infty}(\mathbb R^2)$ is dense in $W^{1,\,p}(\Omega)$ when $\Omega$ is Jordan. 

In this paper, we extend the main idea in \cite{KZ2016} so as to handle both 
multiply connected and higher dimensional settings. It turns out that simply
connectivity (or trivial topology) is not sufficient for approximation results
in higher dimensions.

\begin{thm}\label{example}
Given $1< p<\infty$, there is a bounded domain $\Omega\subset \mathbb R^3,$ 
homeomorphic to the unit ball via a locally bi-Lipschitz homeomorphism, 
such that $W^{1,\,q}(\Omega)$ is not dense in $W^{1,\,p}(\Omega)$ for any $q >p$. 
\end{thm}

Recall that $f:\Omega\to \Omega'$ is locally bi-Lipschitz if
for every compact set $K\subset \Omega$ there exists $L=L(K)$ such that for 
all $x,\,y\in K$
$$\frac 1 L |x-y|\le |f(x)-f(y)|\le L|x-y|. $$

The above example shows that the planar setting is very special. 
The crucial point
is that a simply connected planar domain is conformally equivalent (by
the Riemann mapping theorem) to the unit disk, and conformal equivalence
is in general much more restrictive than topological equivalence. One could
then ask if the planar approximation results extend to hold for those spatial
domains that are conformally equivalent to the unit ball. This is trivially 
the case since the Liouville theorem implies that such a domain is necessarily
a ball or a half-space. A bit of thought reveals that bi-Lipschitz equivalence
is also sufficient. Our results below imply that bi-Lipschitz equivalence can 
be relaxed to quasiconformal equivalence to the unit ball or even to 
quasiconformal equivalence to a uniform domain, a natural class of domains
in the study of (quasi)conformal geometry.

In order to state our main result, we need to introduce some terminology.

\begin{defn}
Let $\Omega\subsetneqq \mathbb R^n$ be a domain. 
Then the associated quasihyperbolic distance between two points 
$z_1,\,z_2\subset \Omega$ is defined as

$$\dist_{qh}(z_1,\,z_2)=\inf_{\gamma} \int_{\gamma} \dist(z,\,\partial \Omega)^{-1}\, dz, $$
where the infimum is taken over all the rectifiable curves $\gamma\subset \Omega$ connecting $z_1$ and $z_2$. A curve attaining this infimum is called a 
quasihyperbolic geodesic connecting $z_1$ and $z_2$. The distance between two 
sets is also defined in a similar manner. 

Moreover, a domain $\Omega$ is called {\it $\delta$-Gromov hyperbolic with respect to the quasihyperbolic metric}, if for all $x,\,y,\,z\in \Omega$ and any 
corresponding quasihyperbolic geodesics $\gamma_{x,\,y},\,\gamma_{y,\,z},\,\gamma_{x,\,z}$,  we have
$$\dist_{qh}(w,\,\gamma_{y,\,z}\cup \gamma_{x,\,z})\le \delta, $$
for any $w\in \gamma_{x,\,y}$. 
\end{defn}

For the existence of quasihyperbolic geodesics we refer to \cite[Proposition 2.8]{BHK2001}. 
For applications, it is usually easier to apply one of the equivalent definitions, see Lemma~\ref{eqdef} below.
Recall that a set $E \subset \mathbb R^n$
is called {\it{quasiconvex}} if there exists a constant $C\ge1$ such that any pair of 
points $z_1,z_2 \in E$ can be connected to each other with a rectifiable curve 
$\gamma \subset E$
whose length satisfies $\ell(\gamma) \le C|z_1-z_2|$.

\begin{thm}\label{mainthm}
If $\Omega\subset \mathbb R^n$ is  a bounded domain that is $\delta$-Gromov hyperbolic with respect to the quasihyperbolic metric, then for any $1\le p< \infty$, $W^{1,\,\infty}(\Omega)$ is dense in $W^{1,\,p}(\Omega)$.  Moreover, if $\Omega$ is also either Jordan or quasiconvex, we have that $C^{\infty}(\mathbb R^n)$ is dense in $W^{1,\,p}(\Omega)$.
\end{thm}

Each finitely connected planar domain is  Gromov hyperbolic with
respect to the quasihyperbolic metric. Therefore we recover the main 
theorem in \cite{KZ2016}. Furthermore, domains which are quasiconformally 
equivalent to uniform domains, especially the ones quasiconformally
equivalent to a ball, are Gromov hyperbolic domains. 
See \cite{BHK2001} for these results. 

Theorem~\ref{mainthm} also gives
consequences for $BV(\Omega)$, 
the Banach space of functions in $L^1(\Omega)$ with bounded variation.
Indeed, given $u\in BV(\Omega)$ we have a sequence of functions 
$u_j\in W^{1,1}(\Omega)$ (or smooth in $\Omega$) that converges to $u$ in 
$L^1(\Omega)$ and
so that the $BV$-energy of $u,$ $||Du||(\Omega),$ satisfies
$$||Du||(\Omega)=\lim_j||\nabla u_j||_{L^1(\Omega)}.$$
Based on Theorem~\ref{mainthm}, we may
further assume that $u_j\in W^{1,\,\infty}(\Omega)$ when $\Omega$ is bounded
and Gromov hyperbolic, and even that each 
$u_j$ is the restriction
of a global smooth function when $\Omega$ is Jordan or quasiconvex. We refer the reader to \cite{AFP} for further information on the theory of $BV$-functions.



The paper is organized as follows. In Section 2 we give some preliminaries. 
After this we decompose a bounded domain $\Omega$ (which is $\delta$-Gromov 
hyperbolic  with respect to the quasihyperbolic metric) into several parts via
Lemma~\ref{eqdef}, and then construct a corresponding partition of unity.  
In \cite{KZ2016} conformal mappings and planar geometry were applied
to obtain the desired composition. In our setting, we cannot rely on mappings 
nor on simple geometry. Instead of this we employ two characterizing 
properties of Gromov hyperbolicity: the ball-separation condition and the 
Gehring-Hayman inequality; see Lemma~\ref{eqdef} below. The proof of  
Theorem~\ref{mainthm} is given in Section 3, and finally in the last section 
we discuss the necessity of  geometric conditions.

The notation in this paper is quite standard. When we make estimates, we often write the constants as 
positive real numbers $C(\cdot)$ with
the parenthesis including all the parameters on which the constant depends. The constant $C(\cdot)$ may
vary between appearances, even within a chain of inequalities.
By $a \sim b$ we mean that $b/C \le a \le Cb$ for some constant $C \ge 2$.
Also $a \lesssim b$ means $a\le C b$ with $C\ge 1$, and similar to $a\gtrsim b$. 
The Euclidean distance between two sets $A,\,B \subset \mathbb R^n$ is denoted by $\dist(A,\,B)$. 
We call a \emph{dyadic cube} in $\mathbb R^n$ any set
\[
 [m_1 2^{-k},\,(m_1 + 1)2^{-k}]\times \cdots \times[m_n2^{-k},\, (m_n+ 1)2^{-k}], 
\]
where $m_1,\,\dots,\, m_n,\,k \in \mathbb Z$.
We denote by $\ell(Q)$ the side length of the cube $Q$, and by $\ell(\gamma)$  the length of a curve $\gamma$. Given a cube $Q$ and $\lambda>0$, by $\lambda Q$ we mean   the cube  concentric with $Q$, with sides parallel to the axes, and with length
$\ell(\lambda Q) =\lambda \ell(Q)$. For a set $A\subset \mathbb R^n$, we denote by $A^o$ its interior, $\partial A$ its boundary, and $\overline A$ its closure. Notation $A\subset\subset B$ means that the set $A$ is compactly contained in $B$.

\section{Decomposition of the domain}
In this section, we first recall some lemmas related to Gromov hyperbolic domains, and then decompose our domain into two main parts. At the end of this 
section we construct a corresponding partition of unity.

Define the {\it{inner distance with respect to $\Omega$}} between 
$x,\,y\in\Omega$ by setting
$$\dist_{\Omega}(x,\,y)=\inf_{\gamma\subset \Omega} \ell(\gamma),$$
where the infimum runs over all curves joining $x$ and $y$ in $\Omega.$ The ball centered at $x$ with radius $r$ respect to the inner distance is denoted by $B_{\Omega}(x,\,r)$.

Let $\Omega\subset \mathbb R^n,$ $n\ge 2$ be a bounded domain that is 
$\delta$-Gromov with respect to the quasihyperbolic metric. 
Recall that $\delta$-Gromov hyperbolicity can equivalently be defined as 
follows; see \cite{BHK2001} and \cite{BB2003}. 

\begin{lem} \label{eqdef}
A domain $\Omega\subset \mathbb R^n$ is  $\delta$-Gromov hyperbolic with respect to the quasihyperbolic metric if and only if it has the following two properties: 
\begin{enumerate}[\rm 1)]
\item  $C_1$-ball-separation condition: There exists a constant $C_1\ge 1$ such that, for any $x,\,y \in \Omega$, any quasihyperbolic geodesic $\Gamma$ joining $x$ and $y$, and every $z\in \Gamma$, the ball $$B=B_{\Omega}(z,\,C_1\dist(z,\,\partial \Omega))$$ 
satisfies   $B\cap \gamma\neq\emptyset$ for any curve $\gamma\subset \Omega$ connecting $x$ and $y$.

\item $C_2$-Gehring-Hayman condition: For any $x,\,y\in \Omega$, the Euclidean length of each quasihyperbolic geodesic connecting $x$ and $y$ is no more than $C_2\dist_{\Omega}(x,\,y)$. 
\end{enumerate}
Here all the constants depend only on each other and $n$. 
\end{lem}

The above Gehring-Hayman condition was proven for simply connected planar
domains in \cite{GH1962} and the ball-separation condition in \cite{BK1995},
respectively.

Recall that every open proper subset of $\mathbb R^n$ admits a Whitney decomposition. A standard reference for this  is \cite[Chapter VI]{stein1970}. 

\begin{lem}\label{Whitney}
Let $\Omega\subsetneqq \mathbb R^n$ be a domain. Then it admits a \emph {Whitney decomposition}, that is, there exists
 a collection $W=\{Q_j\}_{j\in\nn}$ of countably many dyadic (closed) cubes such that

 (i) $\Omega=\cup_{j\in\nn}Q_j$ and $ (Q_k)^\circ\cap (Q_j)^\circ=\emptyset$ for all $j,\,k\in\nn$  with $j\ne k$;

 (ii) $\ell(Q_k)\le \dist(Q_k,\,\partial\boz)\le 4\sqrt n \ell(Q_k)$;

 (iii) $\frac 1 4\ell(Q_k)\le  \ell(Q_j) \le 4\ell(Q_k)$ whenever  $Q_k\cap Q_j\ne\emptyset$.
\end{lem}

The lemmas above allow us to establish the following key lemma. 
\begin{lem}\label{uniform bound}
Suppose $Q_1$ and $Q_2$ are Whitney cubes of $\Omega$ satisfying
$$\frac 1 c \ell(Q_1) \le \ell(Q_2)\le c \ell(Q_1) \quad \text{and } \quad \dist_{\Omega}(Q_1,\,Q_2)\le c \ell(Q_1) $$
for some constant $c>1$. Moreover assume that they can be joined by a chain of 
Whitney cubes, whose edge lengths are larger than $c^{-1}\ell(Q_1)$. 
Then there exists a sequence of no more than $C(c,\,n,\,C_1,\,C_2)$ Whitney 
cubes of $\Omega$, of edge lengths comparable to $\ell(Q_1)$, such that their union connects $Q_1$ and $Q_2$. 
Especially we have 
$$\dist_{qh}(Q_1,\,Q_2)\le C(c,\,n,\,C_1,\,C_2). $$
\end{lem}
\begin{proof}
The $C_2$-Gehring-Hayman condition together with the assumption
$$\dist_{\Omega}(Q_1,\,Q_2)\le c \ell(Q_1)$$
gives a quasihyperbolic geodesic $\gamma$ connecting $Q_1$ and $Q_2$ such that $\ell(\gamma)\lesssim \ell(Q_1).$  Since $\ell(Q_1)\sim \ell(Q_2)$, the diameters of the Whitney cubes intersecting $\gamma$ are uniformly bounded from above by a multiple of $\ell(Q_1)$. 

Moreover, for every Whitney cube $Q$ with $Q\cap \gamma \neq \emptyset$, by the $C_1$-ball-separation condition and the definition of Whitney cubes, any other curve connecting $Q_1$ and $Q_2$ must intersect $(4\sqrt{n}C_1)_{\Omega} Q$. 
On the other hand, by our assumption, there exists a sequence of cubes  connecting $Q_1$ and $Q_2$ with edge lengths not less than $ c^{-1} \ell(Q_1)$. It follows that $\ell(Q)\gtrsim \ell(Q_1).$ 

To conclude, for all $Q\cap\gamma\neq\emptyset$, $\ell(Q)\sim \ell(Q_1)$ with the constant only depending on $n$, $c$, and $C_1$. Since $\ell(\gamma)\ls \ell(Q_1)$
the number of Whitney cubes intersecting $\gamma$ must be bounded by a constant depending only on $C_1$, $C_2$, $n$ and $c$. 
\end{proof}

\subsection{The construction of the core part of $\Omega$}\label{core of omega}

Fix a bounded domain $\Omega$ which is $\delta$-hyperbolic  as in Lemma~\ref{eqdef} with the associated constants $C_1$ and $C_2$.

For any constant $c>0$ and any Euclidean cube or internal metric ball $Q$ centered at $x$, we introduce the notation
$$(c)_{\Omega} Q=\left\{y\in \Omega\mid \dist_{\Omega}(y,\,x)\le c\diam(Q) \right\}; $$
this is a (relatively) closed inner metric ball inside $\Omega$.

Let $m\in\mathbb N$ be large enough such that there is at least one Whitney cube in $\Omega$ whose edge length is larger than $2^{-m}$. Let $\mathcal W$ be the collection of all Whitney cubes of $\Omega$, and $Q_0\in \mathcal W$ be one of the largest ones. Then define $\Omega_{m,\,0}$ to be the path-component of
$$\bigcup_{Q\in \mathcal W,\,\ell(Q)\ge 2^{-m}} Q$$
with $Q_0\subset \Omega_{m,\,0}$, see Figure \ref{fig:omega_m0}.

\begin{figure}
\centering
\includegraphics[width=0.9\textwidth]{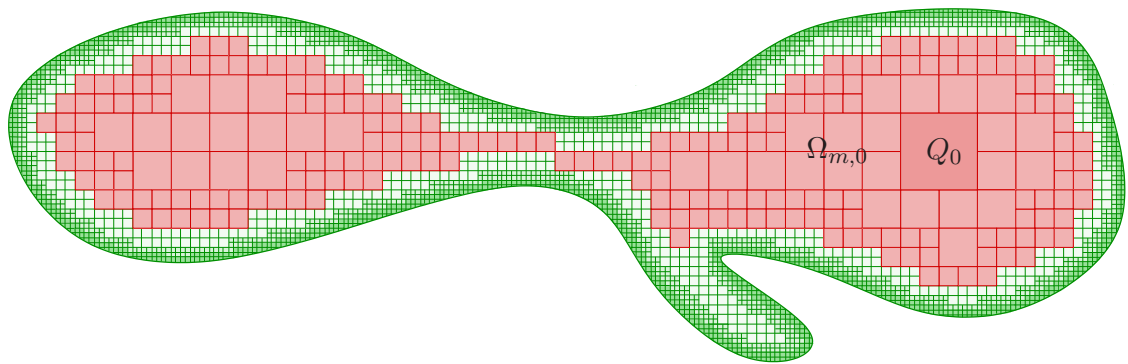}
  \caption{The set $\Omega_{m,\,0}$ is the path-component of the union of cubes of side-length at least $2^{-m}$ that contains $Q_0$.
  In order to have the properties listed in Lemma \ref{core part} for the subdomain $\Omega_m$ we will cut out those parts from $\Omega_{m,\,0}$
  whose connection to $Q_0$ is blocked by dilated boundary cubes.}
  \label{fig:omega_m0}
\end{figure}

Define $\mathcal W_{m,\,0}$ to be the collection of the Whitney cubes in $\mathcal W$ that are contained in $\Omega_{m,\,0}$. 
Also let
$$\mathcal D_{m,\,0}=\left\{Q_i\in \mathcal W \colon  Q_i\subset \Omega_{m,\,0},\, Q_i\cap\partial \Omega_{m,\,0}\neq\emptyset\right\}$$
and 
$$D_{m,\,0}=\bigcup_{Q_i\in \mathcal D_{m,\,0}}Q_i .$$
Notice that, by  definition, any Whitney cube $Q_j\in  \mathcal D_{m,\,0}$ satisfies \begin{equation}\label{equal size}
2^{-m}\le \ell(Q_j)< 2^{-m+2},
\end{equation} 
and thus there are at most finitely many of them since $\Omega$ is bounded. Up to relabeling all the $Q_i$'s in $\mathcal W$ we may assume that all the cubes in $\mathcal D_{m,\,0}$ are ordered consecutively from $1$ to some finite number $N_0$.

Recall the constant $C_1$ in Lemma~\ref{eqdef}. 
We next refine $\Omega_{m,\,0}$ according to the $C_1$-ball separation condition in order to obtain the desired set $\Omega_m$. It is constructed via an induction argument according to the cubes in $\mathcal D_{m,\,0}$.

First for each cube $Q_j\in \mathcal D_{m,\,0}$, we define $U_j=(5\sqrt{n}C_1)_{\Omega}Q_j$. 
Let $m$ be large enough such that $U_j\cap Q_0=\emptyset$. For each $Q_j$ let $Block_j$ (which might be empty) be the union of all the path-components of $\Omega\setminus U_j$ not containing $Q_0$. Roughly speaking, the set $Block_j$ is the collection of points in $\Omega$ whose connection to $Q_0$ is blocked by $U_j$. 
As any curve joining $Q_0$ and some point outside $\Omega_{m,\,0}$ has to pass through $D_{m,\,0}$,  the $C_1$-separation condition allows us to conclude that
\begin{equation}\label{step 0}
\Omega=\Omega_{m,\,0}\cup \bigcup_{Q_j\in \mathcal D_{m,\,0}} U_j \cup \bigcup_{Q_j\in \mathcal D_{m,\,0}} Block_j. 
\end{equation}
Suppose that there exists $Q_k\in \mathcal D_m$ such that $U_j\cap U_k=\emptyset$ and $Block_j\cap U_k\neq \emptyset$. Then by the path-connectedness of $U_k\subset \Omega\setminus U_j$ and the definition of $Block_j$ we conclude that 
\begin{equation}\label{deduction 1}
U_k\subset Block_j. 
\end{equation}

Now let us define 
$$\mathcal W_{m,\,1}=\left\{Q\in\mathcal W_{m,\,0}\colon Q\not\subset \left(Block_1 \setminus (25\sqrt{n}C_1)_{\Omega}Q_1\right) \right\}\subset \mathcal W_{m,\,0}, $$
and  
$$\Omega_{m,\,1}=\bigcup_{Q\in \mathcal W_{m,\,1}} Q \subset \Omega_{m,\,0}. $$
We also define 
$$\mathcal D_{m,\,1}=\left\{Q_i\in \mathcal D_{m,\,0} \colon  Q_i\subset \Omega_{m,\,1},\, Q_i\cap\partial \Omega_{m,\,1}\neq\emptyset\right\}.$$


We claim that
\begin{equation}\label{step 1}
\Omega=\Omega_{m,\,1}\cup \bigcup_{Q_j\in \mathcal D_{m,\,1}} U_j \cup \bigcup_{Q_j\in \mathcal D_{m,\,1}} Block_j.
\end{equation}
Indeed comparing to \eqref{step 0} we have three cases. 

First of all if $y\in Block_k$ with $Q_k\notin \mathcal  D_{m,\,1}$, then 
$$Q_k\subset Block_1\setminus(25\sqrt{n}C_1)_{\Omega}Q_1. $$
This with \eqref{equal size} gives us $U_k\cap U_1=\emptyset$, and consequently $U_k\subset Block_1$ by  \eqref{deduction 1}. Therefore any curve from $y$ to $Q_0$ needs to pass through $U_1$ by the definition of $Block_1$ and the $C_1$-ball-separation condition, and then by definition $y\in Block_{1}$.

Secondly if $y\in U_k$ with $Q_k\notin \mathcal  D_{m,\,1}$, then again
$$Q_k\subset Block_1\setminus(25\sqrt{n}C_1)_{\Omega}Q_1. $$
By the deduction above we similarly conclude that $y\in Block_{1}$.

At last suppose $y\in \Omega_{m,\,0}\setminus\Omega_{m,\,1}$. Then it belongs to some cube $Q$ originally in $\mathcal W_{m,\,0}$ but not in $\mathcal W_{m,\,1}$. Therefore
$$Q\subset Block_1 \setminus (25\sqrt{n}C_1)_{\Omega}Q_1. $$
However $Q$ is connected, and by the argument of \eqref{deduction 1} we also conclude that 
$y\in Block_{1}$. All in all we have shown \eqref{step 1}.


If $Q_2\notin \mathcal D_{m,\,1}$, then we just let $\Omega_{m,\,2}=\Omega_{m,\,1}$ and accordingly define $\mathcal D_{m,\,2}$ and so on. Otherwise, we apply the procedure above, with $Q_1$ replaced by $Q_2$ and $Block_1$ replaced by $Block_2$, to obtain these sets (and collections). We repeat this process for every $Q_j$ with $3\le j\le N_0$. By iteration we finally obtain a set $\Omega_m:=\Omega_{m,\,N_0}$.

Notice that any Whitney cube in $\mathcal W_{m,\,N_0}$ intersecting $\partial \Omega_m$ is contained in $(60\sqrt{n}C_1)_{\Omega}Q_j$ for some $1\le j\le N_0$. Thus it has edge length comparable to $2^{-m}$ with the constant only depending on $n$ and $C_1$. 
Hence there exists a constant $M=M(C_1,\,n)$ such that
$$\Omega_m\subset\subset\Omega_{m'}$$
whenever $m'\ge m+M$. 
The deduction above together with the fact that $\Omega=\bigcup_{m}\Omega_{m,\,0}$ also gives
$$\Omega=\bigcup_{m}\Omega_m.$$
Moreover $\mathcal D_m:=\mathcal D_{m,\,N_0}$ consists of cubes from $\mathcal D_{m,\,0}$. 
To conclude,  we obtain the following lemma.

\begin{lem}\label{core part}
Let $\Omega$ be a bounded  domain which is $\delta$-Gromov hyperbolic with respect to the quasihyperbolic metric, $\mathcal W=\{Q_j\}$ be the collection of Whitney cubes of $\Omega$ and $Q_0$ be one of the largest Whitney cubes. Then there exists a sequence of sets $\Omega_m\subset\subset \Omega$ such that by setting
$$\mathcal B_m=\left\{Q_j\in \mathcal W \colon  Q_j \subset \Omega_{m},\,  Q_j\cap\partial \Omega_{m}\neq\emptyset\right\},$$
by letting $U_j=(5\sqrt{n}C_1)_{\Omega}Q_j$
for  each $Q_j\in \mathcal B_m,$ and by finally defining $Block_j$ 
(which might be empty) to be the union of all the path-components of $\Omega\setminus U_j$ not containing $Q_0$, 
we have the following properties.
\begin{enumerate}[\rm 1)]
\item Each $\Omega_m$ consists of finitely many Whitney cubes and any two of them can be joined by a chain of Whitney cubes in $\Omega$ of edge lengths not less than $2^{-m}$. Moreover $Q_0\subset \Omega_m$ and there exists a constant $M=M(\delta,\,n)$ such that
$$\Omega=\bigcup_{m}\Omega_m,$$
and
$$\Omega_m\subset\subset\Omega_{m'}$$
for any $m'\ge m+M$. 
\item For every Whitney cube $Q_j\in \mathcal B_m$ we have $2^{-m} \le \ell(Q_j)\ls 2^{-m}$. We call such a cube  a {\it boundary cube of $\Omega_m$}. 
\item There exists a subcollection $\mathcal D_m$ of  $\mathcal B_m$ such that for each $Q_k\in \mathcal B_m$ and $Q_j\in \mathcal D_m$
$$Q_k\cap Block_j\neq \emptyset \ \Rightarrow \ Q_k\subset (60\sqrt{n}C_1)_{\Omega}Q_j.$$ 
Moreover $\{(60\sqrt{n}C_1)_{\Omega}Q_j\}_{Q_j\in \mathcal D_m}$ covers all the boundary cubes of $\Omega_m$.  
\item We have
\begin{equation*}
\Omega=\Omega_{m}\cup \bigcup_{Q_j\in \mathcal D_{m}} U_j \cup \bigcup_{Q_j\in \mathcal D_{m}} Block_j. 
\end{equation*}
\end{enumerate}
\end{lem}

The property 3) above turns out to be crucial later and it may fail for
$\mathcal B_m$; this is the reason for introducing the subcollection 
$\mathcal D_m$ of $\mathcal B_m$. 

  \begin{figure}
\centering
\includegraphics[width=0.9\textwidth]{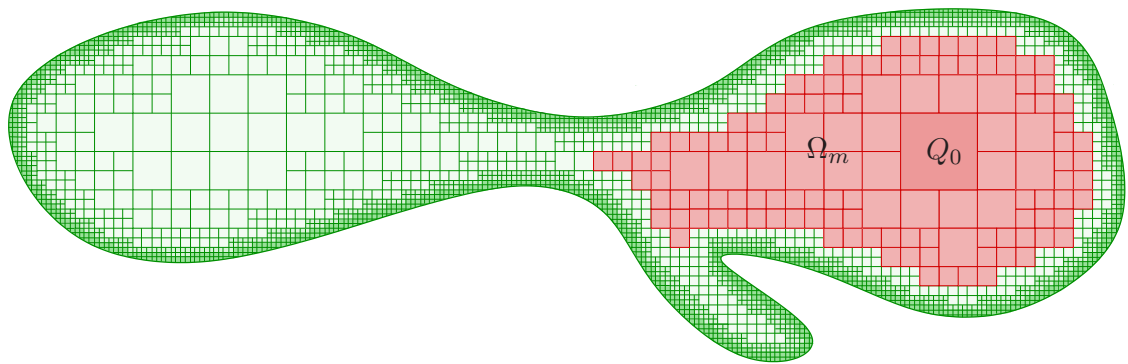}
  \caption{The set $\Omega_m$ obtained after the iterative procedure from sets $\Omega_{m,i}$ still contains the cube $Q_0$.}
  \label{fig:omega_m}
 \end{figure}

\subsection{The decomposition of the boundary layer}

In this subsection we first decompose $\Omega\setminus\Omega_m$ into two main parts $E_m$ and $F_m$, and then make further decompositions of them.
 
First of all let 
$$E_m=\left(\bigcup_{Q_j\in \mathcal D_m} (70\sqrt{n}C_1)_{\Omega} Q_j\right)\setminus \Omega_m. $$
Secondly, we denote by $F_m$ the rest of $\Omega$, that is,
$$F_m=\Omega\setminus(\Omega_m\cup E_m). $$
Notice that by Lemma~\ref{core part} we have
$$\dist_{\Omega}(F_m,\,\Omega_m)\ge 2^{-m}, $$ 
and 
$$F_m\subset \bigcup_{Q_j\in \mathcal D_m} Block_j$$
where the set $Block_j$
 is defined in Lemma~\ref{core part}. 

By abuse of notation, we also denote by  $E_m$ and $F_m$ their closures with respect to the topology of $\Omega$, respectively. 
Observe that  the boundary of $(c)_{\Omega} Q$ in $\Omega$ is porous and hence
of Lebesgue measure zero,
\begin{equation}\label{finite boundary}
|\partial (c)_{\Omega} Q\cap \Omega|=0,
\end{equation}
for each 
$Q\in \mathcal W$ and each $c.$
Therefore we have 
$$|\Omega_m\cap E_m|=|E_m\cap F_m|=0.$$

\subsubsection{The decomposition of $E_m$}

We decompose $E_m$ further. Recall that
$$E_m\subset \bigcup_{Q_j\in \mathcal D_m}(71\sqrt{n}C_1)_{\Omega} Q_j.$$
Let $V_j=(71\sqrt{n}C_1)_{\Omega} Q_j$ for each $Q_j\in \mathcal D_m$. For simplicity we again assume that
$$\mathcal D_m=\{Q_1,\,\dots,\,Q_N\}$$
with some $N\le N_0$.
We claim that for each fixed $V_j$, 
\begin{equation}\label{finite q}
\#\left\{1\le k\le N \colon V_j\cap  V_k\neq \emptyset\right\}\le C(n,\,C_1), 
\end{equation}
where $\#$ means the cardinality of the corresponding set.
Indeed, if  $V_j\cap V_k\neq \emptyset$, then $\dist_{\Omega}(Q_k,\,Q_j)\lesssim 2^{-m}$ by the definition of $V_j$. Then \eqref{finite q} follows by the fact that  $\ell(Q_j)\sim 2^{-m}$ with a constant independent of $j$.

Define $S_1=V_1\cap E_m$, and inductively for $j\ge 2$ set
$$S_j=\left( V_j\setminus \bigcup_{i=1}^{j-1} V_i \right)\cap E_m.$$
Notice that $S_j$ may well be disconnected, or even empty. We replace every $S_j$ by its closure with respect to the topology of $\Omega$, and still use the notation $S_j$. 
Notice that after all these changes,  $V_j,\,S_j$ still satisfy all the corresponding properties above; especially $$E_m\subset \bigcup_{j}S_j.$$
By \eqref{finite q} for each $S_j$ 
\begin{equation}\label{finite s}
\#\left\{1\le k\le N\colon  S_j\cap S_k\neq \emptyset\right\}\le C(n,\,C_1),
\end{equation}
and the corresponding $Q_j,\,Q_k\in \mathcal D_m$ satisfy
$$\dist_{\Omega}(Q_j,\,Q_k)\lesssim 2^{-m}.$$
Similar reasons also give the fact that   
\begin{equation}\label{finite r s}
\#\left\{k\in\mathbb N\colon Q_j\in \mathcal B_m,\, Q_j\cap S_k\neq \emptyset\right\}\le C(n,\,C_1).
\end{equation}
At last we remark that for any $j,\,k$
$$|S_j\cap S_k|=0$$
by \eqref{finite boundary}. Moreover by the definition of $S_j$ we have
\begin{equation}\label{S cap F}
\diam_{\Omega}(S_j)\lesssim 2^{-m}. 
\end{equation}

\subsubsection{The decomposition of $F_m$}

Recall that $Q_0$ is one of the largest Whitney cubes contained in $\Omega_m$, and for each $Q_j\in \mathcal D_m$ we have $U_j=(5 \sqrt n C_1)_{\Omega} Q_j$ and $V_j=(71\sqrt{n}C_1)_{\Omega} Q_j$. 

   \begin{figure}
\centering
\includegraphics[width=0.9\textwidth]{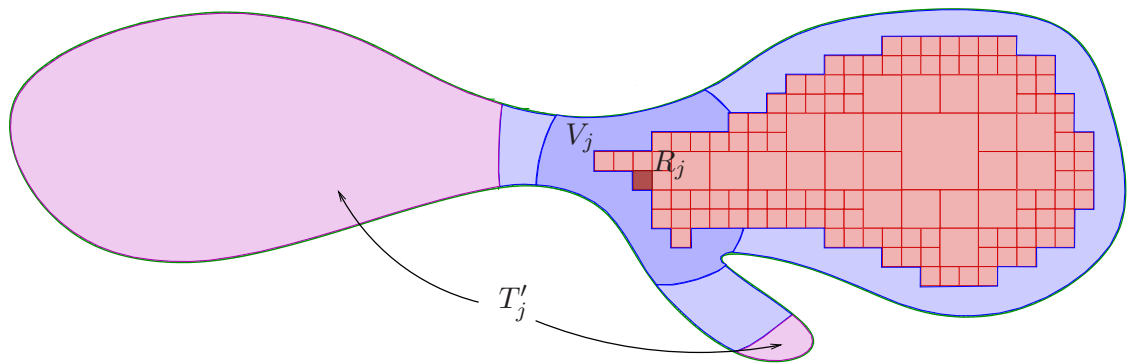}
  \caption{The set $T_j'$ consist of all the path-components of $F_m$ for which the set $U_j$ centered at $Q_j$ blocks all curves going from the path-component to the cube $Q_0$.}
  \label{fig:TjRj}
 \end{figure}

To decompose the last part $F_m$, we introduce the following notation. Recall the definition of $Block_j$ in Lemma~\ref{core part} and  define 
$$T'_j =Block_j\cap F_m. $$
See Figure~\ref{fig:TjRj}. 
Certainly $T'_j$ could be empty. We replace $T'_j$ by its closure with  respect to the topology of $\Omega$ and still denote it by $T'_j$.  Notice that by Lemma~\ref{core part}
$$F_m=\bigcup_{j} T'_j. $$

Fix  $T'_j$ and suppose that $T'_j\cap S_k\neq \emptyset$. We claim that $V_k\cap V_j \neq\emptyset$. Indeed, if $V_k\cap V_j =\emptyset$, then $V_k\cap U_j =\emptyset$ and by the path-connectedness of $V_k$ any point $y\in T'_j\cap V_k$ can be connected to $Q_k\in\mathcal D_m$ by a path in 
$$V_k\subset \Omega\setminus U_j .$$ 
If $Q_k$ can be  connected to $Q_0$ via a path in $\Omega\setminus U_j$, then $y$ can be connected to  $Q_0$ via a path in $\Omega\setminus U_j$, which leads to a contradiction to the definition of $Block_j$, which contains $T'_j$. Then our claim follows. 
If $Q_k$ cannot be connected to $Q_0$ via any path in $\Omega\setminus U_j$, then $Q_k\subset Block_j$, and by Lemma~\ref{core part} we know that
$$\dist_{\Omega}(Q_k,\,Q_j)\le 60\sqrt{n} C_1\diam(Q_j). $$
The claim follows from the definition of $V_j$ .

Therefore by the proof of \eqref{finite q} and the definition of $S_k$ we conclude that, for each fixed (non-empty) $T'_j$, 
\begin{equation}\label{finite t'}
\#\left\{1\le k\le N\colon  T'_j\cap S_k\neq \emptyset\right\}\le C(n,\,C_1). 
\end{equation} 
Also note that if $y\in T'_j\cap T'_k$, then the path-component of $T'_j$ containing $y$ is a subset of $T'_j\cap T'_k$ by the definition of $T'_j$.

We define $T_1=T'_1 $,  and for $j\ge 2$ set
$$T_j=T'_j\setminus \bigcup_{i=1}^{j-1}T'_i. $$
We also refer by $T_j$ to its closure with  respect to the topology of $\Omega$.

According to \eqref{finite t'} for each fixed non-empty $T_j$
\begin{equation}\label{finite t s}
\#\left\{1 \le k\le N\colon  T_j\cap S_k\neq \emptyset\right\}\le C(n,\,C_1). 
\end{equation}
Similarly for each fixed $S_j$
\begin{equation}\label{finite s t}
\#\left\{1\le k\le N\colon  S_j\cap T_k\neq \emptyset\right\}\le C(n,\,C_1).
\end{equation}

To conclude from the subsections above, whenever $Q_j\cap S_k\neq \emptyset$ with $Q_j\in \mathcal B_m$, $S_j\cap S_k\neq \emptyset$ or $ T_j\cap S_k\neq \emptyset$ we always have the corresponding $Q_j,\,Q_k\in \mathcal B_m$ satisfying
$$\dist_{\Omega}(Q_j,\,Q_k)\lesssim 2^{-m}.$$
This fact with Lemma~\ref{core part} allows us to apply Lemma~\ref{uniform bound} later.

\subsection{A partition of unity}

We construct a partition of unity in this subsection. To this end, let us introduce the following notation. For a set $A\subset \Omega$, we define
$$\mathcal N_{m,\,\Omega}(A)=\left\{x\in\Omega \colon \dist_{\Omega}(x,\,A)\le 2^{-m-5}\right\}. $$

\begin{lem}\label{partition}
With all the notations above, there are functions $\psi$, $\phi_j$ and $\varphi_j$ with $1\le j\le N$ such that: 
\begin{enumerate}[\rm 1)]
\item The function $\psi$ is Lipschitz in $\Omega$, compactly supported in $\Omega_m$, $0\le \psi \le1$, and $|\nabla \psi(x)|\lesssim 2^{-m}$.
\item For each $j$, we have $\phi_j\in W^{1,\,\infty}(\Omega)$. The support of $\phi_j$ is relatively closed in $\Omega$ and contained in $\mathcal  N_{m,\,\Omega}(S_j)$, $0\le \phi_j\le1$, and $|\nabla \phi_j|\lesssim  2^{m}$.  
\item For each $j$, we have  $\varphi_j\in W^{1,\,\infty}(\Omega)$. The support of $\varphi_j$ is relatively closed in $\Omega$ and contained in $\mathcal N_{m,\,\Omega}(T_j)$, $0\le \varphi_j\le1$, and $|\nabla \varphi_j|\lesssim  2^{m}$.  
\item $\psi(x)+\sum_j \phi_j(x)+\sum_j \varphi_j(x)= 1$ for any $x\in \Omega$. 
\end{enumerate}
\end{lem}
\begin{proof}
First of all we construct cut-off functions for each of our sets via the distance functions with respect to the inner metric. The function $\phi_j$  can be defined as
$$\phi_j(x)=\max\left\{1-2^{m+6}\dist_{\Omega}(x,\,S_j),\,0\right\}, $$
and similarly
$$\varphi_j(x)=\max\left\{1-2^{m+6}\dist_{\Omega}(x,\,T_j),\,0\right\}. $$ 
 The function $\psi$ is defined by
$$\psi(x)=\min\left\{2^{m+8}\dist_{\Omega}(x,\,E_m\cup F_m),\,1 \right\}. $$
It is obvious that these functions satisfy 
$$\psi(x)+\sum_j \phi_j(x)+\sum_j \varphi_j(x)\ge 1 $$
for every $x\in \Omega$. 

Note that by the essence of \eqref{finite r s}, \eqref{finite s}, \eqref{finite t s} and \eqref{finite s t} we have for each  $Q_j\in \mathcal B_m$
\begin{equation}\label{finite r ns}
\#\left\{1 \le k\le N \colon,\, Q_j\cap \mathcal N_{m,\,\Omega}(S_k)\neq \emptyset\right\}\le C(n,\,C_1),
\end{equation}
and also for each  $Q_j\in \mathcal D_m$
\begin{align}
&\#\left\{1 \le k\le N \colon  \mathcal N_{m,\,\Omega}(S_j)\cap \mathcal N_{m,\,\Omega}(S_k)\neq \emptyset\right\}\le C(n,\,C_1),\label{finite ns}\\
&\#\left\{1 \le k\le N \colon  \mathcal N_{m,\,\Omega}(S_j)\cap \mathcal N_{m,\,\Omega}(T_k)\neq \emptyset\right\}\le C(n,\,C_1),\label{finite ns nt}\\
&\#\left\{1 \le k\le N \colon  \mathcal N_{m,\,\Omega}(T_j)\cap \mathcal N_{m,\,\Omega}(S_k)\neq \emptyset\right\}\le C(n,\,C_1). \label{finite nt ns}
\end{align}
Hence by the decomposition of $\Omega$ we conclude that for any $x\in\Omega$
$$\chi_{\Omega}(x)\le \Phi(x):=\psi(x)+\sum_{j}\phi_j(x)+\sum_{j}\varphi_j(x)\lesssim \chi_{\Omega}(x).$$
Therefore, by  dividing $\psi,\,\phi_j ,\,\varphi_j $ by $\Phi$, respectively, we obtain the desired partition of unity. The new functions, still denoted by $\psi,\,\phi_j ,\,\varphi_j $, satisfy the desired gradient control as $\Phi$ is bounded from below and above. 
\end{proof}
Notice that 
\begin{equation}\label{measure control}
|\mathcal N_{m,\,\Omega}(S_j)\setminus S_j |\ls 2^{-mn}
\end{equation}
 by \eqref{S cap F} uniformly in $j.$  
Moreover the proof of \eqref{finite t'} also shows that there are uniformly finitely many $V_k$ such that  $V_k$ intersects $N_{m,\,\Omega}(T_j)\setminus T_j$. Thus 
\begin{equation}\label{measure control 2}
|\mathcal N_{m,\,\Omega}(T_j)\setminus T_j|\ls 2^{-mn}. 
\end{equation}
Furthermore by an argument similar to the proof of \eqref{finite t'}, for every $1\le j \le N$ we conclude that
\begin{equation}\label{finite nt nt}
\#\left\{1\le k\le N\colon  \mathcal N_{m,\,\Omega}(T_j)\cap \mathcal N_{m,\,\Omega}(T_k)\neq \emptyset\right\}\le C(n,\,C_1). 
\end{equation}

\section{Proof of Theorem~\ref{mainthm}}

\begin{proof}[Proof of Theorem~\ref{mainthm}]
Fix $\ez>0$. Also fix $u\in W^{1,\,p}(\Omega)$ with $1\le p<\infty$. 
We may assume that $u$ is smooth and bounded since bounded smooth functions 
are dense in $W^{1,\,p}(\Omega)$; 
e.g. see the proof of \cite[Lemma 2.6]{KZ2016}. 
We may further assume that $\|u\|_{L^{\infty}(\Omega)}=1$. 

 Recall that $\Omega=\Omega_m\cup E_m \cup F_m$. 
Define $D'_m$ to be the union of those Whitney cubes $Q\in \mathcal W$  for 
which there exists a chain of no more than $M$ Whitney cubes joining $Q$ to 
some cube in $\mathcal B_m.$ Here the constant $M$ that depends on 
$C_1,\,C_2,\,n$ will be determined later.
Then the quasihyperbolic distance from $Q$ to $\cup_{Q\in \mathcal B_m}Q$ is
uniformly bounded if $Q\subset D'_m.$ 
Observe that,  for any Whitney cube  $Q\subset D'_m$ we have
$$\ell(Q)\ls 2^{-m}, $$
with a constant depending on $C_1,\,C_2,\,n$. Also notice that Lemma~\ref{core part} implies $|\Omega\setminus \Omega_m|\to 0$. Thus for $m\in\mathbb N$ large enough we have
\begin{equation}\label{requirement}
\|u\|^p_{W^{1,\,p}(D'_m\cup E_m\cup F_m)}\le \ez \ \text{ and  }\  |E_m\cup F_m|\le \ez. 
\end{equation}

Notice that $u|_{\Omega_m}\in W^{1,\,\infty}(\Omega_m)$ since $\Omega_m$ is compact and $u$ is smooth. 
We define a function $u_m$ on $\Omega$ by setting
$$u_m(x)=u(x)\psi(x)+\sum_{j} a_j \phi_j(x)+\sum_j a_j\varphi_j(x),$$
where $\varphi(x)$, $\phi_j(x)$ and $\psi_j(x)$ are the functions in Lemma~\ref{partition} and $$a_j=\bint_{Q_j} u(x)\, dx$$  is the integral average over $Q_j\in \mathcal B_m$.

It is obvious that $u_m\in W^{1,\,\infty}(\Omega)$ by our construction, since by boundedness of $\Omega$ we only have finitely many $Q_j\in \mathcal B_m$ and Lemma~\ref{partition} gives the estimates on the derivatives. 
Moreover we have $\|u_m\|_{L^{\infty}(\Omega)}\le 1$ by our assumption, Lemma~\ref{partition} and the definition of $u_m$. Hence 
$\|u_m\|_{L^{p}(\Omega\setminus\Omega_m)}\le \ez$  by \eqref{requirement}.
Consequently, by the definition of $\mathcal N_{m,\,\Omega}(\cdot)$ and Lemma~\ref{partition}, we only need to show that 
\begin{equation*}\label{aim for k}
\int_{(\Omega\setminus\Omega_m)\cup (\bigcup_{Q\in \mathcal B_m }Q)}|\nabla u_m|^p\, dx \lesssim \ez. 
\end{equation*}
We will show this via the Poincar\'e inequality, Lemma~\ref{uniform bound} and Lemma~\ref{partition}.

We write $G(Q_j,\,Q_k)$ for the union of the cubes given by Lemma~\ref{uniform bound} for each pair $Q_j,\,Q_k\in  \mathcal B_m$. 
Recall that $\dist_{\Omega}(F_m,\,\Omega_m)\ge 2^{-m}. $ 
Then for any $Q_j\in \mathcal B_m$ with the associated average $a_j$, by \eqref{finite r ns}, \eqref{measure control}, Lemma~\ref{uniform bound}, Lemma~\ref{partition} and the Poincar\'e inequality we obtain that
\begin{align*}
&\int_{Q_j} |\nabla u_m|^p \, dx \lesssim \int_{Q_j} |\nabla (u_m-a_j)|^p \, dx \\
\lesssim& \int_{Q_j} |\nabla [ (u(x)-a_j)\psi(x)]|^p \, dx +\sum_{\substack{S_k\subset E_m\\ Q_j \cap \mathcal  N_{m,\,\Omega}(S_k) \neq\emptyset}} \int_{Q_j} |\nabla [(a_k-a_j) \phi_k(x) ]|^p \, dx \\
\lesssim& \int_{Q_j} |\nabla u|^p + |u(x)-a_j|^p 2^{mp}\, dx +\sum_{\substack{S_k\subset E_m\\ Q_j \cap \mathcal N_{m,\,\Omega}(S_k) \neq\emptyset}} \int_{Q_j} |a_k-a_j|^p 2^{mp} \, dx  \\
\lesssim& \int_{Q_j} |\nabla u|^p\, dx + \sum_{\substack{S_k\subset E_m\\ Q_j \cap \mathcal N_{m,\,\Omega}(S_k) \neq\emptyset}} 2^{-m(n-p)} 2^{-m(p-n)} \int_{G(Q_j,\, Q_k)}|\nabla u|^p \,dx\\
\lesssim& \int_{Q_j} |\nabla u(x)|^p \, dx+ \sum_{\substack{S_k\subset E_m\\ Q_j \cap\mathcal  N_{m,\,\Omega}(S_k) \neq\emptyset}} \int_{G(Q_j,\, Q_k)}|\nabla u|^p \,dx. 
\end{align*}
Notice that by Lemma~\ref{uniform bound} there are uniformly finitely many cubes contained in the chain $G(Q_j,\, Q_k)$ connecting $Q_j$ and $Q_k$ if $\mathcal  N_{m,\,\Omega}(S_k)  \cap Q_j \neq\emptyset$.

On the other hand recall that $\psi(x)$ is compactly supported in $\Omega_m$. Then for each $S_j$, Lemma~\ref{uniform bound}, \eqref{S cap F}, \eqref{finite ns}, \eqref{finite ns nt}, \eqref{measure control} and the Poincar\'e inequality give
\begin{align*}
&  \int_{S_j} |\nabla u_m|^p \, dx \lesssim  \int_{S_j} |\nabla (u_m-a_j)|^p\, dx \\
 \lesssim& \sum_{\substack{S_k\subset E_m\\  \mathcal N_{m,\,\Omega}(S_j)\cap \mathcal N_{m,\,\Omega}(S_k) \neq\emptyset}} \int_{S_j} |\nabla [ (a_k-a_j) \phi_k(x) ]|^p \, dx \\
&\qquad + \sum_{\substack{T_k\subset F_m\\  \mathcal N_{m,\,\Omega}(S_j)\cap \mathcal N_{m,\,\Omega}(T_k) \neq\emptyset}} \int_{S_j} |\nabla [ (a_k-a_j) \varphi_k(x) ]|^p \, dx\\
\lesssim& \sum_{\substack{S_k\subset E_m\\  \mathcal N_{m,\,\Omega}(S_j)\cap \mathcal N_{m,\,\Omega}(S_k) \neq\emptyset}}  \int_{S_j} |a_k-a_j|^p 2^{mp} \, dx + \sum_{\substack{T_k\subset F_m\\  \mathcal N_{m,\,\Omega}(S_j)\cap \mathcal N_{m,\,\Omega}(T_k) \neq\emptyset}} \int_{S_j}  |a_k-a_j|^p 2^{mp} \, dx\\
\lesssim& \sum_{\substack{S_k\subset E_m\\  \mathcal N_{m,\,\Omega}(S_j)\cap \mathcal N_{m,\,\Omega}(S_k) \neq\emptyset}} 2^{-m(p-n)} 2^{-m(n-p)}\int_{G(Q_j,\, Q_k)} |\nabla u|^p  \, dx\\
&\qquad +\sum_{\substack{T_k\subset F_m\\  \mathcal N_{m,\,\Omega}(S_j)\cap \mathcal N_{m,\,\Omega}(T_k) \neq\emptyset}} 2^{-m(p-n)} 2^{-m(n-p)}\int_{G(Q_j,\, Q_{k})} |\nabla u|^p  \, dx\\
\lesssim& \sum_{\substack{S_k\subset E_m\\  \mathcal N_{m,\,\Omega}(S_j)\cap \mathcal N_{m,\,\Omega}(S_k) \neq\emptyset}} \int_{G(R_j,\, R_k)} |\nabla u|^p  \, dx+\sum_{\substack{T_k\subset F_m\\  N_{m,\,\Omega}(S_j)\cap \mathcal N_{m,\,\Omega}(T_k) \neq\emptyset}} \int_{G(Q_j,\, Q_{k})} |\nabla u|^p  \, dx. 
\end{align*}

The calculation for $T_j$ is almost the same. Indeed  by \eqref{finite nt ns}, \eqref{measure control 2}, \eqref{finite nt nt} and the Poincar\'e inequality
\begin{align*}
&  \int_{T_j} |\nabla u_m|^p \, dx \lesssim  \int_{T_j} |\nabla (u_m-a_j)|^p\, dx \\
 \lesssim& \sum_{\substack{S_k\subset E_m\\  \mathcal N_{m,\,\Omega}(T_j)\cap \mathcal N_{m,\,\Omega}(S_k) \neq\emptyset}} \int_{T_j} |\nabla [ (a_k-a_j) \phi_i(x) ]|^p \, dx \\
&\qquad + \sum_{\substack{T_k\subset F_m\\  \mathcal N_{m,\,\Omega}(T_j)\cap \mathcal N_{m,\,\Omega}(T_k) \neq\emptyset}} \int_{T_j} |\nabla [ (a_k-a_j) \phi_i(x) ]|^p \, dx  \\
\lesssim& \sum_{\substack{S_k\subset E_m\\  \mathcal N_{m,\,\Omega}(T_j)\cap \mathcal N_{m,\,\Omega}(S_k) \neq\emptyset}}  \int_{T_j} |a_k-a_j|^p 2^{mp} \, dx +\sum_{\substack{T_k\subset F_m\\  \mathcal N_{m,\,\Omega}(T_j)\cap \mathcal N_{m,\,\Omega}(T_k) \neq\emptyset}}  \int_{T_j} |a_k-a_j|^p 2^{mp} \, dx \\
 \lesssim&  \sum_{\substack{S_k\subset E_m\\  \mathcal N_{m,\,\Omega}(T_j)\cap \mathcal N_{m,\,\Omega}(S_k)\neq\emptyset}} 2^{-m(p-n)} 2^{-m(n-p)}\int_{G(Q_{j},\, Q_j)} |\nabla u|^p  \, dx \\
&\qquad + \sum_{\substack{T_k\subset F_m\\  \mathcal N_{m,\,\Omega}(T_j)\cap \mathcal N_{m,\,\Omega}(T_k)\neq\emptyset}} 2^{-m(p-n)} 2^{-m(n-p)}\int_{G(Q_{j},\, Q_j)} |\nabla u|^p  \, dx\\
 \lesssim& \sum_{\substack{S_k\subset E_m\\  \mathcal N_{m,\,\Omega}(T_j)\cap \mathcal N_{m,\,\Omega}(S_k)\neq\emptyset}} \int_{G(Q_{j},\, Q_k)} |\nabla u|^p  \, dx + \sum_{\substack{T_k\subset F_m\\  \mathcal N_{m,\,\Omega}(T_j)\cap \mathcal N_{m,\,\Omega}(T_k)\neq\emptyset}} \int_{G(Q_{j},\, Q_k)} |\nabla u|^p  \, dx. 
\end{align*}

By Lemma~\ref{uniform bound}, there is a constant $C_3=C_3(C_1,\,C_2,\,n)$ such that, for any chain of cubes $G(Q_{j},\, Q_k)$ used above the number of cubes involved is uniformly bounded from above by $C_3$. This gives us the constant $M$
in the definition of $D'_m$. 

Sum over all the $Q_j$'s, $S_j$'s and $T_j$'s above. Notice that, since the number of Whitney cubes in any chain $G(Q_j,\,Q_k)$ above is always uniformly 
bounded by Lemma~\ref{uniform bound}, the  Whitney cubes involved in our
sums have uniformly finite overlaps. 
Additionally all the cubes in these chains are contained in $D'_m$. 
Thus we obtain \eqref{requirement} and  conclude the first part of the theorem. 

When $\Omega$ is quasiconvex, we immediately have that $C^{\infty}(\mathbb R^n)$ is dense in $W^{1,\,p}(\Omega)$ since every function in $W^{1,\,\infty}(\Omega)$ can then be extended to a global Lipschitz function;  by applying suitable 
cut-off functions and via a diagonal argument we obtain the 
approximation by smooth functions.

The argument for the Jordan domain case is similar to the proof of \cite[Corollary 1.2]{KZ2016}. 
Recall that for any two non-empty subsets $X$ and $Y$ of $\mathbb R^n$, the  {\it Hausdorff distance} $\dist_{\mathrm H}(X,\, Y)$ is defined as
$$\dist_{\mathrm H}(X,\,Y) = \max\{\,\sup_{x \in X} \inf_{y \in Y} d(x,\,y),\, \sup_{y \in Y} \inf_{x \in X} d(x,\,y)\}. $$

When $\Omega$ is Jordan, we can construct a 
sequence 
of Lipschitz domains $\{G_s\}_{s=1}^{\infty}$ approaching $\Omega$ in Hausdorff 
distance such that $\Omega\subset \subset G_{s+1}\subset \subset G_{s}$
and
$$\dist_{\mathrm H}(G_s,\,\partial \Omega)\le 2^{-s}$$
 for each $s\in\mathbb N$. 
For example, by the Morse-Sard theorem we may define $G_s$ via the boundary of 
a suitable lower level set of $d$, where $d$ is a smooth function obtained by 
applying suitable mollifiers and a partition of unity for 
$\mathbb R^n\setminus \Omega$ to the distance function $\dist(x,\, \Omega)$.

Now  fix $m\in \mathbb N$ and choose $s$ such that $s\ge 2m$. Then, by the 
definition of $G_s,$  the $2C_1$-separation condition with respect to $G_s$ 
holds for our original cubes in $\Omega_m$. Similarly for points with 
inner distance smaller than a multiple of $2^{-m}$ in $\Omega_m$, 
the $2C_2$-Gehring-Hayman condition with respect to $G_s$ still  holds. 
Moreover, the original Whitney cubes contained in $\Omega_m$ are also  
Whitney-type for $G_s$ up to a multiplicative constant $2$ in
 Lemma~\ref{Whitney}. 
 Therefore we may  repeat all the arguments above similarly to extend the 
function $u_m$ from $\Omega_m$ to $v_m\in W^{1,\,p}(G_s)$, with
$$\|u-v_m\|_{W^{1,\,p}(\Omega)}\ls \ez. $$
Since each $G_s$ is a Lipschitz domain, we may extend $v_m$ to a global 
Sobolev function, and then by applying suitable mollifiers and via a 
diagonal argument we obtain the approximation by global smooth functions. 
\end{proof}

\section{Proof of Theorem~\ref{example}}
When $n\ge3$, unlike in the planar case, simply connectivity does not 
guarantee that $W^{1,\,\infty}(\Omega)$ be dense in $W^{1,\,p}(\Omega)$ for 
$1\le p<\infty$.  
Indeed, given $1<p<\infty$ there exists a simply connected bounded domain 
$\Omega\subset \mathbb R^3$ such that even $W^{1,\,p}(\Omega)$ is not dense in 
$W^{1,\,q}(\Omega)$ for $1\le q<p$. 

Towards this, let us recall the definition of removable sets. A closed set $E\subset \mathbb R^n$ with Lebesgue measure zero is said to be {\it removable for $W^{1,\,p}$} if $$W^{1,\,p}(\mathbb R^n)=W^{1,\,p}(\mathbb R^n\setminus E)$$
in the sense of sets. In \cite[Theorem A]{K1999}, for any $1< p\le n$, Koskela gave an example  of  a compact set $E\subset \mathbb R^n$ which is removable for $W^{1,\,p} $ but not for $W^{1,\,q} $ with $1\le q<p$. We give a related planar example for every $1<p<\infty$. 

\begin{thm}\label{remove}
Let $1<p<\infty$. Then there is a compact set $E\subset \mathbb R^2$ of Lebesgue measure zero  such that  $E$ is removable for $W^{1,\,q} $ when $p< q<\infty$ but not for $W^{1,\,q} $ when $1\le q\le p$.
\end{thm}

By taking the union of a suitable collection of scaled and translated
copies $E_i$ of the above compact sets corresponding to an increasing sequence
of $p_i$ tending to a fixed $p$ we obtain the following corollary.

\begin{cor} Let $1<p<\infty.$ Then there is a compact set 
$E\subset \mathbb R^2$ of Lebesgue measure zero  such that  $E$ is removable 
for  $W^{1,\,q} $  when $q\ge p$ but not for $W^{1,\,q} $ when $1\le q<p.$
\end{cor}

We divide the proof of Theorem~\ref{remove} into two lemmas. 

\begin{lem}\label{koskela}
Let $1<p\le 2$. Then there is a compact set $E\subset \mathbb R^2$ of Lebesgue measure zero such that  $E$ is removable for $W^{1,\,q} $ when $p< q<\infty$ but not for $W^{1,\,q} $ when $1\le q\le p$.
\end{lem}
\begin{proof}
The proof essentially follows from the proof of \cite[Theorem A]{K1999}. 

We first consider the case where $1<p<2$. By  \cite[Proposition 2.1, Theorem 2.2, Theorem 2.3]{K1999} it suffices to construct a Cantor set $E\subset [0,\,1]$ of positive length so that, by letting $I_j$ be the complementary intervals of $E$ on $[0,\,1]$ and $\mathscr H^1$ the $1$-dimensional Hausdorff measure,
$$\sum_{j=1}^{\infty}\mathscr H^1(I_j)^{2-p}<\infty, $$
while $E$ is $q$-porous for all $p<q\le 2$. Recall that $E\subset [0,\,1]$ is 
$q$-porous if for $\mathscr H^1$-almost every point $x\in E$ there is a sequence of numbers $r_i$ and a constant $C_x$ such that $r_i\to 0$ as $i\to \infty$, and each interval $[x-r_i,\,x+r_i]$ contains an interval $I_i\subset [0,\,1]\setminus E$ with $\mathscr H^1(I_i)\ge C_x r_i^{\frac 1 {2-q}}$. 

Towards this construction, we let $0<s<\frac 1 3$ be a small constant to be determined momentarily. Out set $E$ is obtained via the following Cantor construction. At the $i$-th step with $i\in \mathbb N$ we delete an open interval of length $si^{-\frac 2 {2-p}} 2^{-\frac {i+1} {2-p}}$ from the middle of  each of the remaining $2^i$ closed intervals with equal length, respectively. Then E is defined as the intersection of all these closed intervals, and $s$ is chosen such that
$$\sum_{i}s 2^ii^{-\frac 2 {2-p}} 2^{-\frac {i+1} {2-p}}<1. $$
Thus $E$ has positive length, and it is not difficult to check that $E$ has the desired properties. 

When $p=2$ we similarly construct $E$ by removing intervals of length $s2^{-i}\exp(-2^i)$ with sufficiently small (and fixed) $s$ at $i$-th step. Then by the proof of \cite[Theorem A]{K1999} and \cite[Theorem 3.1]{K1999}, $E$ is not $q$-removable for any $1<q\le 2$. The removability of $E$ for $q>2$ comes from \cite[Proposition 2.1]{K1999}  again. 
\end{proof}

\begin{lem}\label{tapio}
Let $2<p<\infty$. Then there is a compact set $E\subset \mathbb R^2$ of Lebesgue measure zero such that  $E$ is removable for $W^{1,\,q} $ when $p< q<\infty$ but not for $W^{1,\,q} $ when $1\le q\le p$.
\end{lem}
\begin{proof} 
We separate our proof into three steps. 

\noindent\textbf{Step 1: The construction of $E$. }
The set $E$ is defined as a product set $C\times F$, where $C\subset \mathbb R$ is a Cantor set of Hausdorff dimension less than $1$ and $F\subset \mathbb R$ is a Cantor set with positive Lebesgue measure, called a {\it fat Cantor set}. 

Let us start with the construction of $C$. Given a sequence $\{\lambda_i\}_{i\in \mathbb N_+}$ with $0<\lambda_i<\frac 1 2$, we build a symmetric Cantor set  with  $\lambda_i$ as the contraction ratio at step $i$. More precisely, define
\[
 C = \bigcap_{i=0}^\infty C_i,
\]
where $C_0 = I_{0,1} = [0,1]$ and $C_i$ with $i\ge 1$ are  defined iteratively as follows: 
When $I_{i,j} = [a,b]$ has been defined, let $I_{i+1,2j-1} = [a,a+\lambda_i|a-b|]$
and $I_{i+1,2j} = [b-\lambda_i|a-b|,b]$. This is well-defined as $\lambda_i<\frac 1 2$. 
Then we set 
\[
C_i = \bigcup_{j=1}^{2^i} I_{i,j}.
\]

For the fat Cantor set $F$, likewise we associate it with a sequence of positive real numbers $(\beta_i)_{i=1}^\infty$  such that
$$\beta_i=(1-\lambda_{i+1})P_i$$
where $P_i=\lambda_1\lambda_2\cdots\lambda_i$, and $\lambda_i$ are from the previous paragraph. Clearly 
\begin{equation}\label{beta}
\sum_{i=1}^\infty \beta_i < 1
\end{equation}
 as $\lambda_i<1$. The numbers $\beta_i$ denote the lengths of the disjoint open intervals removed from the unit interval. To be more specific, we define the approximating sequence $F_i \subset \mathbb R$ with respect to $\beta_i$ in the following way. Let $F_0 = [0,1]$. Then iteratively, at step $n$, we replace one of the largest remaining intervals $[a,b]$ of the set $F_{n-1}$ by the set
\begin{equation}\label{F}
 \left[a,a+r\right] \cup \left[a+r + \beta_n,b\right], \quad \text{where} \quad
 r = \frac12\left(b-a-\sum_{i=n}^\infty \beta_i\right)
\end{equation} 
and obtain $F_n$ in this way. We claim that, there is always one interval in $F_{n}$ that has length strictly larger than $\sum_{i=n+1}^\infty \beta_i$. If so, then   $F$ is  well-defined. 

 Indeed when $n=0$ our claim follows immediately \eqref{beta}. Then under the induction assumption that there is an interval $[a,\,b]\subset F_{n-1}$ satisfying $b-a>\sum_{i=n}^\infty \beta_i$, we further have  that at the $n$-th step by \eqref{F} there is an interval with length
$$b-a-r-\beta_n=\frac 1 2 \left(b-a-\beta_n+\sum_{i=n+1}^\infty \beta_i\right)> \sum_{i=n+1}^\infty \beta_i,$$ 
where the last inequality comes from the induction assumption. Therefore the largest interval in $F_{n}$ has length strictly larger than $\sum_{i=n+1}^\infty \beta_i$. Thus our claim follows. Moreover the length of the largest remaining interval in $F_n$ goes to zero as $n\to \infty$ by \eqref{F}. Thus $F$ is a topological Cantor set. 
The fact that $\mathscr H^{1}(F)>0$ comes from \eqref{beta}. 
\\

\noindent\textbf{Step 2: The unremovability of $E$ for $q\le p$. }
Fix $p>2$ and a set $E=C\times F$ from the Step 1, with the sequence $\{\lambda_i\}$ to be determined later. Let $1\le q\le p$. We define a function $v\in W^{1,\,q}(\mathbb R^2\setminus E)$ such that it cannot be extended to $W^{1,\,q}(\mathbb R^2)$. 
To do this, we first construct a function $u\in L^{\infty}(\mathbb R^2\setminus E)$ with $\nabla u\in L^{q}(\mathbb R^2\setminus E)$.

Let $u(x,\,y)=0$ if $x< 0$ and $u(x,\,y)=1$ if $x>1$. 
For each $y\in F$ further define
$$u(x,\,y)=\frac {2j-1}{2^{i+1}}\quad \text{ for  $\ x\in I_{i,\,j}\setminus (I_{i+1,\,2j-1}\cup I_{i+1,\,2j})$}, $$ 
where $ i\in \mathbb N,\,1\le j\le 2^i.$ 
Then for $y\in F$, $u(x,\,y)$ is a Cantor step function with respect to $x$ if we extend it continuously. In Figure~\ref{fig:strips} we give an example of such a function.

Next we define $u(x,\,y)$ for $y\not\in F$. For $(x,\,y)\in [0,\,1]^2\setminus E$ and $\dist(y,\,F)\le \dist(x,\,C)$ we also set 
$$u(x,\,y)=\frac {2j-1}{2^{i+1}}\quad \text{ for  $\ x\in I_{i,\,j}\setminus (I_{i+1,\,2j-1}\cup I_{i+1,\,2j})$}, $$ 
where $ i\in \mathbb N,\,1\le j\le 2^i.$  
Then for fixed $y_0\notin F$, on the horizontal line $y=y_0$ we  have already defined the function $u$ up to finitely many open intervals. We then simply define $u$ as an affine function on each remaining interval so that it is continuous on this line. 
Then $u$ is defined in $\mathbb R^2\setminus E$,
and the set
$$\left\{ (x,\,y)\colon u(x,\,y)=\frac {2j-1}{2^{i+1}} \right\}$$
has Lipschitz boundary.

We claim that $u$ is also continuous in $\mathbb R^2\setminus E$. Indeed if $\dist(y,\,F)< \dist(x,\,C)$, then by definition $u$ is locally constant and hence certainly continuous. For the remaining case where $0\neq \dist(y,\,F)\ge \dist(x,\,C)$, there is an open interval $I$ such that $y\in I$, $I\cap F=\emptyset$, and for every $y_0\in I$ the function $u(x,\,y_0)$ is Lipschitz with the constant depending only on $\dist(y_0,\,F)$ (as $E$ is already fixed). Then for any such an $x$, in the vertical direction $u$ is also continuous since 
the affine-extension is done with respect to domains where $u$ is locally constant and whose boundary is a $1$-Lipschitz graph. Consequently, $u$ is even locally Lipschitz. Hence $u$ is a continuous function. 

\begin{figure}
  \centering
  \includegraphics[width=0.9\textwidth]{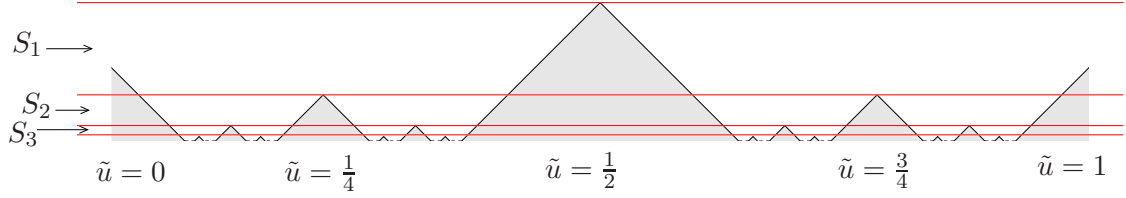}
  \caption{An illustration of the definition of $u$ for $\lambda_i = \frac13$.
  On the $x$-axis the function $u$ is defined as the Cantor step function.
  The constant regions are extended to the complement as shown by the gray areas. For any horizontal line
  there are finitely many open intervals where the function is not defined by the previous extension. 
  On each such interval we extend it as an affine function. We then estimate $|\nabla u|$
  separately on each horizontal strip $S_i$.}
  \label{fig:strips}
 \end{figure}

We next estimate the Sobolev-norm of $u$. First up to a suitable translation we consider $u$ in a strip $S_i$ which is defined as
\[
 S_i = \mathbb R \times \left[\frac12(1-\lambda_{i+1})P_i, \frac12(1-\lambda_{i})P_{i-1} \right],
\]
and is a part of $\mathbb R^2\setminus E$ (up to a suitable translation). Also recall that $P_i=\lambda_1\lambda_2\cdots\lambda_i.$ 
Then each $S_i$ minus the triangles where the function is defined as constant has at most $2^i$ connected components $K$.

Up to another translation, each component $K$ equals 
\[
 \left\{(x,y) \in \mathbb R^2\colon |x| < y ,\,\frac12(1-\lambda_{i+1})P_i < y < \frac12(1-\lambda_{i})P_{i-1}\right\}
\]
and up to adding a constant the function $u$ restricted on it is defined as
\[
 \tilde u(x,y) = 2^{-i-1}\frac{x}{y};
\]
see Figure~\ref{fig:stripf}.
Thus $|\nabla \tilde u| \lesssim 2^{-i}P_i^{-1}$
in the strip $S_i$.   
Since each of the $2^i$ components $K$ have width and height comparable to $P_i$, we get
\[
\int_{S_i}|\nabla \tilde u|^q \lesssim 2^iP_i^2 \left(2^{-i}P_i^{-1}\right)^q = 2^{i(1-q)}P_i^{2-q}.
\]

  \begin{figure}
   \includegraphics[width=0.45\textwidth]{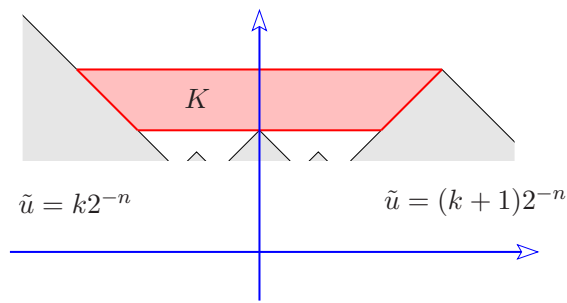}
  \caption{Each strip $S_i$  consists of $2^i$ connected components outside the regions where $\tilde u$ was initially defined
  as constant. Such a component $K$ is drawn here in the case $\lambda_i = \frac13$. Also the choice of the coordinates
  used in the estimate is indicated in the figure.}
  \label{fig:stripf}
 \end{figure}

Let us recall that $\beta_i = (1-\lambda_{i+1})P_{i}$.  This implies that we only have copies of $S_i$ in $\mathbb R\times(F_{j-1}\setminus F_j)$ with $i\ge j$. 
Consequently we  have
\begin{equation}\label{eq:nablaestimate}
 \int_{\mathbb R^2\setminus E}|\nabla u|^q \lesssim \sum_{j=1}^\infty\sum_{i=j}^\infty2 \int_{S_i}|\nabla \tilde u|^q \lesssim   \sum_{i=1}^\infty2i \int_{S_i}|\nabla \tilde u|^q \lesssim \sum_{i=1}^\infty i2^{i(1-q)}P_i^{2-q}. 
\end{equation}
By H\"older's inequality and the fact that $E$ is compact, it suffices to check the non-removability for the case $q = p$.

Choose $\lambda_i$ in such a way that
\begin{equation}\label{eq:Pidef}
  2^{i(1-p)}P_i^{2-p} = \frac{1}{i^3}
\end{equation}
for all $i$ large enough. That is,
\[
 \lambda_i =\min\left\{ \frac 1 3,\,\frac{P_i}{P_{i-1}}\left(\frac{(i-1)^32^{(i-1)(1-p)}}{i^32^{i(1-p)}}\right)^\frac{1}{2-p}\right\}
  =\min\left\{ \frac 1 3,\, 2^\frac{p-1}{2-p} \left(\frac{i}{i-1}\right)^\frac{3}{p-2}\right\}.
\]
Observe that $\lambda_i\sim 2^\frac{p-1}{2-p}$ with the constant independent of $i$. 
With this choice
\begin{align*}
 \sum_{i=1}^\infty i2^{i(1-p)}P_i^{2-p}
 =  \sum_{i=1}^\infty \frac{1}{i^2} < \infty.
\end{align*}
Therefore by \eqref{eq:nablaestimate} we conclude that $\nabla u\in L^{p}(\mathbb R)$. 

By letting
$$v=u\varphi,$$
where $\varphi\in C^{\infty}_c(\mathbb R^2)$ has support in $[-1,\,2]^2$ and 
satisfies $\varphi(x)=1$ for $x\in [0,\,1]^2$, we 
have $v\in W^{1,\,p}(\mathbb R^2\setminus E)$. 
However $v$ cannot be extended to a function in $W^{1,\,p}(\mathbb R^2).$ 
Indeed, by the Sobolev embedding theorem for $p>2,$
the precise representative of an extension $w\in W^{1,\,p}(\mathbb R^2)$ 
would continuous, while by definition the extension of $v$ is a 
Cantor function (multiplied by a smooth function) when restricted to $y=y_0$ 
for $y_0\in F$ with $|F|>0$. This would contradict the fact that the precise representative of a Sobolev function is absolutely continuous
 along almost every line parallel to the coordinate axes; see \cite[4.5.3, 4.9.2]{EG1992}. 
\\

\noindent\textbf{Step 3: The removability of $E$ for $q> p$.} We claim that, for the set $E$ defined  above, for every two points $z_1,\,z_2\in \mathbb R^2\setminus E$ there is a curve $\gamma \subset \mathbb R^2\setminus E$ such that
\begin{equation}\label{curve}
\int_{\gamma}\dist(z,\,E)^{\frac 1 {1-q}}\, ds(z)\le C(p,\,q) |z_1-z_2|^{\frac {q-2}{q-1}}. 
\end{equation}
If so, then by \cite[Theorem 1.1]{S2010} (or by \cite{K1998}), we conclude that   any function in $W^{1,\,q}(\mathbb R^2\setminus E)$ can be extended to $W^{1,\,q}(\mathbb R^2)$. Since the Lebesgue measure of $E$ is zero, it follows that  
$W^{1,\,q}(\mathbb R^2\setminus E)=W^{1,\,q}(\mathbb R^2)$ and hence $E$ is removable for $W^{1,\,q}(\mathbb R^2)$. 

Now let us show the claim. We only consider the case where 
$z_1,\,z_2\in [0,\,1]^2$, as the general case can be easily reduced to it. 
 For any $z_1,\,z_2\in [0,\,1]^2\setminus E$, we write $z_1=(x_1,\,y_1)$ and $z_2=(x_2,\,y_2)$. First we may assume that $y_1,\,y_2\notin F$. Indeed if $y_1\in F$ then $x_1\notin C$. Then there is a removed interval $I\subset [0,\,1]$ (in the construction of $C$) containing $x_1$. Find a point $x\in I$ such that 
$$3|x-x_1|\le \min\{|z_1-z_2|,\,\diam(I)\} \ \text{ and } \  3\dist(x,\,C)\ge  \min\{|z_1-z_2|,\,\diam(I)\}; $$
the existence of such an $x$ follows from the triangle inequality. Let $w_1=(x,\,y_1)$. 
Next since $F$ is topologically a Cantor set, as $y_1\in F$ one can find a point $z_1'=(x,\,y_1')$ such that $3|y_1-y_1'|\le  |z_1-z_2|$ and $y_1'\notin F$. Then the curve consisting of the two line segments $[z_1,\,w_1]$ and $[w_1,\,z_1']$ satisfies
$$\int_{[z_1,\,w_1]\cup [w_1,\,z_1']}   \dist(z,\,E)^{\frac 1 {1-q}}\, ds(z) \lesssim  \int_{0}^{|z_1-z_2|} t^{\frac 1 {1-q}}\, dt+|z_1-z_2| ^{\frac { 2-q} {1-q}}\lesssim |z_1-z_2| ^{\frac { q-2} {q-1}}, $$
with the constant depending only on $q$.
We may also apply a  similar argument for $z_2$. Thus our assumption is legitimate. 

Under such an assumption  we are going to construct the curve connecting $z_1,\,z_2$. Recall that $\lambda_i\sim 2^\frac{p-1}{2-p}$ and $P_i=\lambda_1\lambda_2\cdots \lambda_i$. Then there is a natural number $n$ such that $P_{n+1}\le |z_1-z_2|\le P_n$. Notice that there is an interval $I_0 \in \{I_{n,\,k}\}_{k=1}^{2^n}$ such that 
$$\max\{\dist(I_0,\,z_1),\,\dist(I_0,\,z_2)\}\lesssim  |z_1-z_2|\ \text{ and } \ \diam(I_0)=P_n$$
with the constant depend only on $p$ by the Cantor construction. Denote by $x_0$ the middle point of such an interval. Let $\gamma=[z_1,\,(x_0,\,y_1)]\cup[(x_0,\,y_1),\,(x_0,\,y_2)]\cup [(x_0,\,y_2),\,z_2]$ be the curve joining $z_1,\,z_2$ and consisting of three line segments; see Figure~\ref{fig:curve}. We show that $\gamma$ is the desired curve. 

 \begin{figure}
  \includegraphics[width=0.35\textwidth]{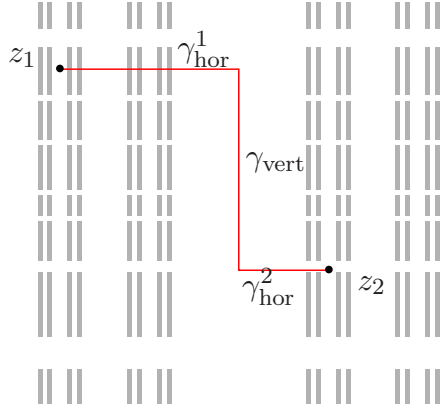}
  \caption{We connect any two points $z_1, z_2 \in \mathbb R^2 \setminus E$ (up to a negligible error near the end-points) with a curve consisting of a vertical part $\gamma_\text{vert}=[(x_0,\,y_1),\,(x_0,\,y_2)]$ 
  and two horizontal parts $\gamma_\text{hor}^1= [z_1,\,(x_0,\,y_1)]$, $\gamma_\text{hor}^2=[(x_0,\,y_2),\,z_2]$.
  The desired estimate on the vertical part comes from the almost self-similarity of the Cantor set $C$ with dimension strictly less than $1$,
  whereas for the horizontal parts we have to make a bit more careful estimate.}
  \label{fig:curve}
 \end{figure}

In fact for the vertical part $[(x_0,\,y_1),\,(x_0,\,y_2)]$, as $x_0$ is the middle point of $I_0$ with $\diam(I_0)=P_n$ and $\lambda_i\sim 2^\frac{p-1}{2-p}$, we have
$$\int_{[(x_0,\,y_1),\,(x_0,\,y_2)]} \dist(z,\,E)^{\frac 1 {1-q}}\, ds(z)\lesssim  |z_1-z_2| ^{\frac { q-2} {q-1}}, $$
with the constant depending only on $p$ and $q$. Hence it suffices for us to consider the horizontal ones. 

First of all
\begin{align*}
\int_{[z_1,\,(x_0,\,y_1)]} \dist(z,E)^\frac{1}{1-q} \, ds(z)
&\lesssim \sum_{i=n}^\infty 2^{i-n} \int_0^{P_i}t^\frac{1}{1-q}\ dt
\lesssim \sum_{i=n}^\infty 2^{i-n}P_i^{\frac{q-2}{q-1}}\\
& \lesssim |z_1-z_2|^\frac{q-2}{q-1}\sum_{i=n}^\infty 2^{i-n}\left(\frac{P_i}{P_n}\right)^{\frac{q-2}{q-1}}.
\end{align*}
Therefore we are left with estimating the last sum in the above expression. 
This sum is bounded from above independently of $n$, since
\begin{align*}
\sum_{i=n}^\infty 2^{i-n}\left(\frac{P_i}{P_n}\right)^{\frac{q-2}{q-1}} & =
\sum_{i=n}^\infty 2^{i-n}\left(\frac{i^\frac{3}{p-2}2^{-i\frac{p-1}{p-2}}}{n^\frac{3}{p-2}2^{-n\frac{p-1}{p-2}}}\right)^{\frac{q-2}{q-1}}
 = \sum_{i=n}^\infty \left(\frac{i}{n}\right)^{\frac{3}{p-2}\cdot\frac{q-2}{q-1}} 2^{(i-n)\left(1-\frac{p-1}{p-2}\cdot \frac{q-2}{q-1}\right)} \\
&\le \sum_{i=1}^\infty i^{\frac{3}{p-2}\cdot\frac{q-2}{q-1}} 2^{i\left(1-\frac{p-1}{p-2}\cdot \frac{q-2}{q-1}\right)} < \infty,
\end{align*}
where we have used the assumption $q>p$ to have convergence of the last sum via the fact that
$\frac{p-1}{p-2}\cdot \frac{q-2}{q-1} > 1.$ The estimate for $[(x_0,\,y_2),\,z_2]$ is similar. Hence we have shown the claim, and then the second part of the theorem follows. 
\end{proof}

\begin{proof}[Proof of Theorem~\ref{example}]
Let 
$$\Omega=A\times (0,\,0.5]\cup (-1,\,2)^2 \times (0.5,\,1):=((-1,\,2)^2\setminus E) \times (0,\,0.5] \cup (-1,\,2)^2 \times (0.5,\,1), $$
where $E\subset  (-1,\,2)^2$ is compact and removable for $W^{1,\,q} $ for all 
$p< q<\infty$ but not for $W^{1,\,p}$. Such a set $E$ exists by 
Theorem~\ref{remove} (scale and translate if necessary).

Let $u(x,\,y)\in W^{1,\,p}(A)$, and $\hat u(x,\,y,\,z)=u(x,\,y)\kappa(z)$ for $0\le z\le 1$, where $\kappa(z)$ is a smooth function with $\kappa(z)=1$ if $0\le z\le \frac 1 4$, $0 \le \kappa \le1$, $|\nabla \kappa(z)|\le 10$ and $\kappa(z)=0$ if $\frac 3 4\le z\le 1$.  By definition $\hat u \in W^{1,\,p}(\Omega)$.

Note that  removability is a local question. Namely $E$ is removable for $W^{1,\,p}$ if and only if for each $x\in E$ there is $r>0$ such that
$$W^{1,\,p}(B(x,\,r)\setminus E)=W^{1,\,p}(B(x,\,r)); $$
see e.g. \cite{K1999}. 
Hence if $\hat u$ can be approximated by $\{\hat u_n\}$ in the $W^{1,\,p}$-norm with $\hat u_n\in W^{1,\,q}(\Omega)$, then by Fubini's theorem and the fact that $E$ is removable for $W^{1,\,q}$, for almost every $0\le z\le \frac 1 4$ we get a sequence, denoted by $u_n\in W^{1,\,q}((-1,\,2)^2)\subset W^{1,\,p}((-1,\,2)^2)$, approaching some $\wz u$ in $W^{1,\,p}((-1,\,2)^2)$. Note that $\wz u$ coincides with $u$ on $A$. 
This then contradicts the unremovability of $E$ since we chose $u$ arbitrarily; notice that $E$ has $2$-Lebesgue measure zero.

We finally show that $\Omega$ is homeomorphic to a ball via a locally bi-Lipschitz map. Towards this, for $w=(x,\,y,\,z)\in \Omega$ define
$$f_1(w)=f_1(x,\,y,\,z)=(x,\,y,\,z \dist( w,\, E\times (0,\,0.5] )) $$
for $w=(x,\,y,\,z)\in \Omega.$
Then $f_1$ is locally bi-Lipschitz, and $f_1$ is a homeomorphism as it fixes 
the first two coordinates and is a homeomorphism with respect to the third 
one. Moreover, $f_1(\Omega)$ is a Lipschitz domain as the bottom of $\Omega$ is mapped to a square in the $xy$-plane and $f_1$ bi-Lipschitz on the rest of
the boundary of $\Omega.$ Hence there is another (locally) 
bi-Lipschitz homeomorphism 
$f_2$ mapping $f_1(\Omega)$ onto the unit ball. 
Letting $f=f_2\circ f_1$ we conclude 
that $\Omega$ is locally bi-Lipschitz homeomorphic to a ball. 
\end{proof}

\end{document}